\documentclass[12pt]{amsart}
\usepackage{multirow}
\usepackage{fancyvrb}

\usepackage{amssymb,graphicx,amsthm}
\usepackage{amsmath}
\usepackage{makecell}

\def\S{{\bf S}}
\def\A{{\bf A}}

\def\T{{\bf T}}
\def\L{{\bf L}}
\def\N{{\bf N}}
\def\J{{\bf J}}

\def\AI{\A_{\mathcal I}}
\def\AH{\A_{\mathcal H}}
\def\I{{\mathcal I}}
\def\H{{\mathcal H}}

\newtheorem{thm}{Theorem}[section]
\newtheorem{lemma}[thm]{Lemma}

\newtheorem{prop}[thm]{Proposition}

\begin{document}

\title{{\bf Computational enumeration of Andr\'e planes}}

\author{Jeremy M. Dover}
\address{1204 W. Yacht Dr., Oak Island, NC 28465 USA}
\email{doverjm@gmail.com}

\begin{abstract}
In this paper, we address computational questions surrounding the enumeration of non-isomorphic Andr\'e planes for any prime power order $q$. We are particularly focused on providing a complete enumeration of all such planes for relatively small orders (up to 125), as well as developing computationally efficient ways to count the number of isomorphism classes for other orders where enumeration is infeasible. Andr\'e planes of all dimensions over their kernel are considered.
\end{abstract}

\maketitle

\section{Introduction}
In their seminal paper, Bruck and Bose~\cite{bruckbose} proved that every finite translation plane can be obtained from a construction that starts with a {\em spread} of an odd-dimensional projective space. Letting that space be $PG(2n+1,q)$, a spread is a partition of the space into $q^{n+1}+1$ pairwise disjoint $n$-dimensional subspaces. Much of this work was anticipated in 1954, by Andr\'e~\cite{andre} who provided a similar construction of finite translation planes using a vector space model that was later shown to be equivalent to the projective space model developed by Bruck and Bose. 

Without getting too rigorous at this point, the basic concept behind Andr\'e's construction starts with a regular spread $\S$ of $PG(2n+1,q)$. Bruck~\cite{bruck:cg2} has shown that for any pair of spaces in $\S$, it is possible to coordinatize $PG(2n+1,q)$ in a manner such that the elements of $\S$ are in one-to-one correspondence with the field $GF(q^{n+1})$, together with a special element $\infty$. In this model, the two starting spaces are $0$ and $\infty$, and the remaining spaces can be partitioned into $q-1$ norm surfaces, where each norm surface consists of a set of spaces that correspond to elements of $GF(q^{n+1})$ with constant norm over $GF(q)$. The key observation is that each of these norm surfaces admits $n$ {\em replacements}, that is sets of $n$-spaces that do not lie in $\S$ but cover exactly the same point set as the spaces in the norm surface of $\S$. By deleting one or more norm surfaces of $\S$ and adding in a replacement for each, one obtains a spread called an Andr\'e spread. The coordinatization of $\S$ in this manner is important, and creates what is called a {\em linear set} of norm surfaces. There may exist sets of pairwise disjoint norm surfaces in a regular spread that do not all lie in a linear set; replacement of these norm surfaces creates a so-called subregular spread, but we do not address those spreads here.

As described by Johnson, et al~\cite{jjb}, the Andr\'e planes are all ``known", in the sense that for any order $q$ it is straightforward to construct all Andr\'e planes; there are no new Andr\'e planes to discover. But with the revolution in computing since the Andr\'e planes were discovered, it is now possible to construct and database models for all known translation planes of order up to and including 128, and the feasibility of order 169 is probably not far off. While it is theoretically possible to construct all $2^{12}$ Andr\'e spreads from a regular spread of $PG(3,13)$, the computational power required to determine isomorphism between more than 4000 projective planes of order 169 is certainly formidable. Our goal in this paper is to develop algorithms using group actions to determine representatives for all isomorphism classes of Andr\'e planes, as well as more efficient algorithms that count isomorphism classes non-constructively with Burnside's lemma.

\section{Some preliminary results about regular spreads}

In this section, we provide a framework for addressing the Andr\'e planes obtained from $PG(2n+1,q)$. Many of the results here are well known for the case $n=1$, and for the higher-dimensional cases many of these results were certainly known to Bruck~\cite{bruck:cg2}, if not explicitly stated. However, there is value in clearly articulating the results here in a unified manner, so that we may refer to them as needed in what follows.

Let $n \ge 1$ be an integer and $q>2$ a prime power. Define $F = GF(q)$, $K = GF(q^{n+1})$, and let $F^*$ denote the set of non-zero elements of $F$, with $K^*$ defined analogously. Let $V = K \oplus K$ be a $(2n+2)$-dimensional vector space over $F$.  This vector space $V$ is a model for $PG(2n+1,q)$ using homogeneous coordinates, such that $(x,y)$ and $(fx,fy)$ define the same point for all $f \in F^*$.

A {\em spread} of $PG(2n+1,q)$ is a set of $q^{n+1}+1$ pairwise disjoint $n$-dimensional subspaces which partition the points of $PG(2n+1,q)$. A {\em regulus} $R$ in $PG(2n+1,q)$ is a set of $q+1$ pairwise disjoint $n$-dimensional subspaces such that any line that meets three spaces in $R$ meets all of them. Bruck and Bose~\cite{bruckbose} show that any three pairwise disjoint $n$-dimensional subspaces lie in a unique regulus, and define a {\em regular spread} to be one that contains the regulus generated by any three of its elements. Regular spreads are significant in that for any $q>2$, the Bruck/Bose construction of translation planes yields a Desarguesian plane if and only if the spread used to construct the plane is regular. From this definition, it is easy to see that any two regular spreads that meet in at least three subspaces must meet in at least an entire regulus; Bruck and Bose show that a regulus and a single space disjoint from all elements of the regulus uniquely determine a regular spread, which shows that $q+1$ spaces is the largest possible intersection between two distinct regular spreads.

Bruck and Bose give a specific coordinate model for a regular spread which will be very useful in what follows. Define $J(\infty) = \{(x,0):x \in K\}$, and for any $k \in K$, let $J(k) = \{(kx,x):x \in K\}$. These are all $n$-dimensional subspaces, and are pairwise disjoint, hence we have a spread $\S = \{J(\infty)\} \cup \{J(k):k \in K\}$; Bruck and Bose prove that this spread is regular.

Let $Aut(K/F)$ denote the group of field automorphisms of $K$ with fixed field $F$, which necessarily has order $n+1$, and let $\sigma$ be any element of $Aut(K/F)$. It is clear by linearity of $\sigma$ that the sets $J^\sigma(k) = \{(kx^\sigma,x):x \in K\}$ are also $n$-dimensional subspaces of $PG(2n+1,q)$. Andr\'e~\cite{andre} noted that for any $f \in F^*$, the set of subspaces $\J_f = \{J(k): \N(k) = f\}$, where $\N$ is the norm function from $K$ to $F$, can be replaced by the set $\J^\sigma_f = \{J^\sigma(k):\N(k) = f\}$, for any $\sigma$ to create a new spread of $PG(2n+1,q)$ that is not regular. Moreover, each $\J_f$ can be replaced (or not) independently as $f$ varies over $F^*$, yielding $(n+1)^{q-1}$ different, though potentially isomorphic, spreads. These spreads are called {\em Andr\'e spreads}, and the translation planes that arise from them are called {\em Andr\'e planes}. The sets $\J_f$, and any set of spaces in $\S$ isomorphic to them, are called {\em norm surfaces} in $\S$.

An important notion when dealing with Andr\'e spreads is the concept of a {\em linear set} of norm surfaces. Bruck has developed an abstract definition of linearity, but ultimately a set $T$ of norm surfaces in $\S$ is linear if there exists a collineation of $PG(2n+1,q)$ which leaves $\S$ invariant and maps $T$ onto a set of norm surfaces $\J_f$. We will shortly recall some of Bruck's results about linearity which require the abstract definition to prove, but we only need the cited results in what follows.

For the $n=1$ case, Bruck~\cite{bruck} extensively developed an isomorphism between the regular spread $\S$ of $PG(3,q)$ and the Miquelian inversive plane $M(q)$. For the higher-dimensional cases, Bruck~\cite{bruck:cg2} generalized this connection to higher-dimensional circle geometries, keeping intact the isomorphism between a regular spread of $PG(2n+1,q)$ and $CG(n,q)$. The key takeaway from the circle geometry connection is that it shows an isomorphism between the regular spread $\S$ of $PG(2n+1,q)$ and the projective line $PG(1,q^{n+1})$, coordinatized as $K \cup \{\infty\}$, where the spaces in $\S$ map to the points of $PG(1,q^{n+1})$, and the reguli in $\S$ map to order $q$ sublines in $PG(1,q^{n+1})$. Moreover, there is a homomorphism $\Psi$ from the group of collineations of $PG(2n+1,q)$ leaving $\S$ invariant to $P\Gamma L(2,q^{n+1})$; this allows us to provide an explicit description of the group of collineations of $PG(2n+1,q)$ leaving $\S$ invariant. This result is fairly well-known, but stating it explicitly allows us to set up some notation.

\begin{prop}
\label{Sgroup}
Let $\S = \{J(\infty)\} \cup \{J(k):k \in K\}$ be a regular spread in $PG(2n+1,q)$. Then the group of collineations of $PG(2n+1,q)$ that leaves $\S$ invariant consists of the transformations $\left\{ \chi_{a,b,c,d,\rho}: \rho \in Aut(K), a,b,c,d \in K, ad-bc \ne 0\right\}$ acting via $$(x,y) \chi_{a,b,c,d,\rho} = (ax^\rho+cy^\rho,bx^\rho+dy^\rho)$$
\end{prop}

\begin{proof}
First note that $\chi = \chi_{a,b,c,d,\rho}$ is a non-singular semi-linear transformation on $V$, and thus is a collineation of $PG(2n+1,q)$; semilinearity is clear from the definition, and non-singularity follows from the condition that $ad-bc \neq 0$. For any $(x,0) \in J(\infty)$, $(x,0)^\chi = (ax^\rho,bx^\rho) \in J(a/b)$ (with the usual conventions on $\infty$), so $\chi$ maps $J(\infty)$ onto $J(a/b)$. If $(kx,x) \in J(k)$, then
$$(kx,x)^\chi = (a(kx)^\rho + cx^\rho, b(kx)^\rho + dx^\rho) = ((ak^\rho+c)x^\rho,(bk^\rho+d)x^\rho)$$
from which it follows that $\chi$ maps $J(k)$ onto $J(\frac{ak^\rho+c}{bk^\rho+d})$. Thus all of these collineations map $\S$ onto itself.

Consider the homomorphism $\Psi$. The kernel of this homomorphism is the group of collineations of $PG(2n+1,q)$ which leaves each element of $\S$ individually invariant. For any $\psi \in Ker(\Psi)$ we must have $J(\infty)\psi = J(\infty)$ and $J(0)\psi = J(0)$, which together imply $\psi$ must be of the form $(x,y)\psi = (xT,yU)$ for semi-linear transformations $T,U$ on $V$. Moreover, $J(1)\psi=J(1)$ implying $U=T$. Letting $k$ be any element of $K$, we have $J(k)\psi = J(k)$, so $(k,1)\psi = (kT,1T) = (ky,y)$ for some $y \in K^*$. This forces $\psi = \chi_{k,0,0,k,1}$ for some $k \in K\*$. It is not difficult to see that each such transformation fixes each element of $\S$, so we find that the kernel of our homomorphism has order $\frac{q^{n+1}-1}{q-1}$.

Therefore the stabilizer group of the regular spread $\S$ in $PG(2n+1,q)$ has order $\frac{q^{n+1}-1}{q-1} \times |P\Gamma L(2,q^{q+1})|$, and it is not hard to calculate that there are the same number of collineations represented by the $\chi_{a,b,c,d,\rho}$, making them the entirety of the stabilizer group.
\end{proof}

Reguli and norm surfaces are the same when $n=1$, but in higher dimensions, the isomorphism between the regular spread $\S$ of $PG(2n+1,q)$ and $PG(1,q^{n+1})$ maps norm surfaces onto structures in $PG(1,q^{n+1})$ that Bruck calls {\em covers}. Bruck~\cite{bruck:cg2} shows that every cover satisfies an equation of the form $\N(x-c) = f$ or $\N\left(\frac{x-a}{x-b}\right) = f$, where $a,b,c \in K$, $a \ne b$ and $f \in F^*$. For fixed $a,b$, or for fixed $c$, the set of covers obtained by varying $f$ over $F^*$ is a {\em linear set} of covers, which corresponds to a linear set of norm surfaces of $\S$. The two points not covered by these sets ($\infty,c$ in the first case, $a,b$ in the second) are called the {\em carriers} of the linear set.

Bruck has proven the following result about linear sets, which are collected here to illustrate an important difference between the $n=1$ and higher-dimensional cases:

\begin{prop}
\label{linear}
Let $\J$ be a norm surface in a regular spread $\S$ of $PG(2n+1,q)$ with $n>1$. Then $\J$ has a unique pair of carriers, and lies in a unique linear set of norm surfaces. If $n=1$, then the lines of $\S$ off $\J$ are partitioned into pairs of carriers of $\J$, and any norm surface (regulus) disjoint from $\J$ lies together with $\J$ in a unique linear set.
\end{prop}

With this result in place, we can prove an important result about how two norm surfaces/covers can intersect; the $n=1$ case follows quickly from Bruck and Bose, while the author proved the $n=2$ case in~\cite{thesis}.

\begin{prop}
\label{intersection}
Let $N_1$ and $N_2$ be distinct norm surfaces in a regular spread $\S$ of $PG(2n+1,q)$, each consisting of $\frac{q^{n+1}-1}{q-1}$ elements of the spread. Then $N_1$ and $N_2$ can intersect in at most $2\frac{q^n-1}{q-1}$ elements of the spread.
\end{prop}

\begin{proof}
For ease of notation, let $\mu = \frac{q^{n+1}-1}{q-1}$ and let $\nu = \frac{q^n-1}{q-1}$. As discussed above, determining the number of spread elements in which two norm surfaces in the regular spread $\S$ intersect is equivalent to determining the size of the intersection of two covers $C_1$ and $C_2$ in $PG(2,q^{n+1})$. Since all covers are isomorphic, we may assume that $C_1$ has equation $\N(x) = 1$. We assume $C_2$ has equation $\N(x-a) = f \N(x-b)$ for some $a \ne b \in K$, $f \in F^*$; the argument for the case where $C_2$ has equation $\N(x-c) = f$ is nearly identical and slightly easier. Suppose $y \in C_1 \cap C_2$. We know $y \in K^*$ and $y^\mu = 1$ since $y \in C_1$, and $(y-a)^\mu = f (y-b)^\mu$ since $y \in C_2$. Note that this latter equation has exactly $\mu$ solutions since $C_2$ is a cover, and thus is not identically zero.

Expand the second equation to yield:
$$(y^{q^n}-a^{q^n})(y-a)^\nu = f(y^{q^n}-b^{q^n})(y-b)^\nu$$

Now multiply on both sides by $y^\nu$, which is not identically zero, to obtain:
$$(y^\mu-a^{q^n}y^\nu)(y-a)^\nu = f(y^\mu-b^{q^n}y^\nu)(y-b)^\nu$$

Since $y^\mu = 1$, we can simplify to obtain:
$$(1-a^{q^n}y^\nu)(y-a)^\nu = f(1-b^{q^n}y^\nu)(y-b)^\nu$$

This equation yields a polynomial expression in $y$ of degree at most $2\nu$ satisfied by all elements of $C_1 \cap C_2$, and is not identically zero. Thus there are at most $2\nu$ values for $y$ satisfying this equation, proving the result.
\end{proof}

A second key difference between the $n=1$ and higher-dimensional cases is the number of potential replacement sets for a norm surface. When sorting isomorphism classes, it becomes important to understand what is happening with all potential replacements under collineations of $PG(2n+1,q)$. To this end, recall that for any $\sigma \in Aut(K/F)$ we have $J^\sigma(k) = \{(kx^\sigma,x):x \in K\}$ is an $n$-dimensional subspace, and our replacements for $\J_f$ are the sets $\J^\sigma_f = \{J^\sigma(k): \N(k) = f\}$. Define $\S^\sigma = \{J(\infty)\} \cap \{J^\sigma(k):k \in K\}$, and let $\lambda_\sigma$ be the collineation defined via $(x,y)\lambda_\sigma = (x,y^\sigma)$. Clearly $\lambda_\sigma$ leaves $J(\infty)$ invariant and maps $J^\sigma(k)$ onto $J(k)$, hence each of the $\S^\sigma$ is a regular spread, and there is a collineation of $PG(2n+1,q)$ that maps $\S$ to $\S^\sigma$ that has the net effect of replacing all of the spaces of the norm surfaces $\J_f$ in $\S$ with $\J^\sigma_f$.

Our final general result describes the stabilizer groups associated with various collections of the norm surfaces in ${\mathcal J} = \{\J^\sigma_f: \sigma \in Aut(K/F), f \in F^*\}$.

\begin{thm}
\label{lingroups}
In $PG(2n+1,q)$, where $n \ge 1$ and $q > 2$ is a prime power, and $(n,q) \ne (1,3)$, let $\S$ be the regular spread $\{J(\infty)\} \cup \{J(k):k \in K\}$ with linear set $\L = \{\J_f: f \in F^*\} \subset {\mathcal J}$. Then 
\begin{enumerate}
\item The collineation group of $PG(2n+1,q)$ leaving the set $\L$ of norm surfaces invariant consists of the collineations $G_\L = \{\chi_{a,0,0,d,\mu}:a,d \in K^*, ad \neq 0, \mu \in Aut(K)\} \cup \{\chi_{0,b,c,0,\mu}:b,c \in K^*, bc \neq 0, \mu \in Aut(K)\}$;
\item the permutation action of $G_\L$ acting on the sets $\J_f \in L$ is action-isomorphic to the group of transformations $\Xi = \{\xi^\pm_{\alpha,\tau}:\alpha \in F^*,\tau \in Aut(F)\}$ acting via $(\J_f)^{\xi^\pm_{\alpha,\tau}} = \J_{\alpha f^{\pm\tau}}$; 
\item the collineation group of $PG(2n+1,q)$ leaving the set ${\mathcal J}$ of norm surfaces invariant consists of the collineations $G_{\mathcal J} = \{\chi\lambda_\sigma:\chi \in G_\L, \sigma \in Aut(K/F)\}$; and
\item the permutation action of $G_{\mathcal J}$ acting on the sets $\J^\sigma \in {\mathcal J}$ is action-isomorphic to the group of transformations $\Upsilon = \{\upsilon^\pm_{\alpha,\tau,\sigma}:\alpha \in F^*,\tau \in Aut(F),\sigma \in Aut(K/F)\}$ acting via $(\J^\rho_f)^{\upsilon^\pm_{\alpha,\tau,\sigma}} = \J^{\rho^{\pm 1}\sigma^{-1}}_{\alpha f^{\pm\tau}}$.
\end{enumerate}
\end{thm}
\begin{proof}
Suppose $\phi$ is a collineation of $PG(2n+1,q)$ that maps $\L$ onto itself. Since $\phi$ maps $q^2-1 > q+1$ spaces of $\S$ onto spaces of $\S$, $\phi$ must leave $\S$ invariant. So by Proposition~\ref{Sgroup} $\phi = \chi_{a,b,c,d,\mu}$ for some $a,b,c,d \in K$, $ad-bc \neq 0$ and $\mu \in Aut(K)$. Moreover, since $\phi$ leaves $\S$ invariant and $\L$ invariant, it must leave $\{J(\infty),J(0)\}$ invariant, meaning $\phi$ must either leave both $J(\infty)$ and $J(0)$ invariant, or interchange them. In the former case, we must have $b=c=0$, while in the latter we must have $a=d=0$, hence $\phi$ must be in the set of collineations $G_\L$.

Let $\chi = \chi_{a,0,0,d,\mu} \in G_\L$. If $(kx,x) \in J(k)$, then $(kx,x)^\chi = (ak^\mu x^\mu,dx^\mu) \in J(ak^\mu/d)$. Note that $\N(a/d) = \alpha$ for some $\alpha \in F^*$, and there exists $\tau \in Aut(F)$ such that $f^\mu = f^\tau$ for all $f \in F$. So if $\N(k) = f$, then we have $\N(ak^\mu/d) = \alpha f^\tau$. $\alpha$ and $\tau$ depend only on $a$, $d$, and $\mu$, so for all $f \in F^*$ we must have $(\J_f)^\chi = \J_{\alpha f^\tau}$ showing that $\chi$ leaves $\L$ invariant. The calculation for $\chi = \chi_{0,b,c,0,\mu}$ is similar, and these results together show that $G_\L$ is the entire group of collineations of $PG(2n+1,q)$ leaving $\L$ invariant. Moreover, these calculations show half of the claimed action isomorphism, namely that every element of $G_\L$ acts on $\L$ as some $\xi^\pm_{\alpha,\tau}$. To show that all such elements can be obtained, note that $\chi_{a,0,0,1,\tau}$ acts as $\xi^+_{\alpha,\tau}$ for any $a \in K^*$ with $\N(a) = \alpha$, and $\chi_{0,b,1,0,\tau}$ acts as $\xi^-_{\alpha,\tau}$ for any $b \in K^*$ with $\N(b) = \alpha$.

Let $H$ be the group of collineations of $PG(2n+1,q)$ that leave ${\mathcal J}$ invariant, and let $\phi \in H$. Suppose first that $n=1$. In this case, since $q > 3$, ${\mathcal J}$ contains at least three reguli lying in $\S$ and at least three reguli lying in $\S^q$. Thus $\phi$ must map at least two reguli $R_1$ and $R_2$ of $\S$ onto reguli of either $\S$ or $\S^q$. Since $R_1$ and $R_2$ are disjoint reguli, they contain $2q+2 > q+1$ lines, forcing $\phi$ to map $\S$ onto either itself or onto $\S^q$; the analogous statement for $\S^*$ implies $\phi$ leaves the set of regular spreads $\{\S,\S^q\}$ invariant. If $n > 1$, the argument is similar, but slightly easier and works for $q > 2$. Since any norm surface has more than $q+1$ spaces it lies in a unique regular spread, hence if $\phi$ maps $\J^\rho_f$ to $\J^\mu_{f'}$, it must map $\S^\rho$ to $\S^\mu$; thus $\phi$ must leave the set of regular spreads $\{\S^\sigma:\sigma \in Aut(K/F)\}$ invariant.

Since the group of all collineations of $PG(2n+1,q)$ leaving $\S$ invariant has order $|G_\L|$, the stabilizer of $\S$ in $H$ is certainly no bigger. We also know that the orbit of $\S$ under $H$ has size at most $n+1$, hence $|H|$ is at most $(n+1)|G_\L|$. Now, consider $\chi \lambda_\sigma \in G_{\mathcal J}$, where $\chi \in G_\L$ and $\sigma \in Aut(K/F)$, and first look at the case where $\chi = \chi_{a,0,0,d,\mu}$ for $a,d \in K^*$ and $\mu \in Aut(K)$. Since $d$ is non-zero, we can can write $N(a/d^\rho) = \alpha$ for some $\alpha \in F^*$. Looking at an arbitrary point of $J^\rho(k)$, we have
\begin{align}
(kx^\rho,x)^{\chi \lambda_\sigma} &= (a k^\mu x^{\rho \mu},d x^\mu)^{\lambda_\sigma} \nonumber\\
                                                    &= (a k^\mu x^{\rho \mu},d^\sigma x^{\sigma \mu}) \nonumber\\
                                                    &= \left(\frac{a k^\mu}{d^\rho}(d^\sigma x^{\sigma \mu})^{\rho \sigma^{-1}},d^\sigma x^{\sigma \mu}\right) \nonumber
\end{align}
This shows that $\chi \lambda_\sigma$ maps every space $J^\rho(k)$ to $J^{\rho \sigma^{-1}}(\frac{a k^\mu}{d^\rho})$. For all $k$ with fixed norm $\N(k) = f$, we have $\N(\frac{a k^\tau}{d^\rho}) = \alpha f^\mu$; moreover, since $\mu$ is an automorphism of $K$, there exists an automorphism $\tau \in Aut(F)$ such that $f^\mu = f^\tau$ for all $f \in F$. Thus we find $\chi \lambda_\sigma$ maps $\J^\rho_f$ to $\J^{\rho \sigma^{-1}}_{\alpha f^\tau}$ for some $\alpha \in F^*$, $\tau \in Aut(F)$, and thus leaves ${\mathcal J}$ invariant.

Now we address the case where $\chi = \chi_{0,b,c,0,\mu}$ for some $b,c \in K^*$ and $\mu \in Aut(K)$. Here $c$ is non-zero, so we can write $N(b/c^{\rho^{-1}}) = \alpha$ for some $\alpha \in F^*$. Again we calculate
\begin{align}
(kx^\rho,x)^{\chi \lambda_\sigma} &= (bx^\tau,ck^\tau x^{\rho \tau})^{\lambda_\sigma} \nonumber\\
                                            &= (b x^\tau,c^\sigma k^{\sigma\tau} x^{\rho \sigma \tau}) \nonumber\\
                                            &= \left(\frac{b}{c^{\rho^{-1}} k^{\rho^{-1}\tau}}(c^\sigma k^{\sigma\tau} x^{\rho\sigma\tau})^{\rho^{-1}\sigma^{-1}},c^\sigma k^{\sigma\tau} x^{\rho\sigma\tau}\right) \nonumber
\end{align}
Using calculations similar to the above, and again letting $\tau \in Aut(F)$ such that $f^\tau = f^\mu$ for all $f \in F$, we see that $\chi \lambda_\sigma$ maps $\J^\rho_f$ to $\J^{\rho^{-1} \sigma^{-1}}_{\alpha f^{-\tau}}$, and thus also leaves ${\mathcal J}$ invariant.

We have shown that $G_{\mathcal J} \subseteq H$. To show equality, note that $G_{\mathcal J}$ ostensibly has $(n+1)|G_\L|$ elements, but there could in principle be some collapsing wherein $\chi_1 \lambda_{\sigma_1} = \chi_2 \lambda_{\sigma_2}$. But if this occurs, we have $\chi_1 = \chi_2 \lambda_{\sigma_1 \sigma_2^{-1}}$. Both $\chi_1$ and $\chi_2$ leaves $\S$ invariant, which forces $\lambda_{\sigma_1 \sigma_2^{-1}}$ to leave $\S$ invariant as well, but this only happens if $\sigma_1 = \sigma_2$, in which case $\chi_1 = \chi_2$. Thus $|G_{\mathcal J}| = (n+1)|G_\L|$, proving it is the entirety of $H$.

The above calculations substantially show the action isomorphism of part 4 as well, since it shows that every $\chi \lambda_\sigma$ acts on ${\mathcal J}$ in the fashion of $\upsilon^\pm_{\alpha,\tau,\sigma}$ for some $\alpha \in F^*$, $\tau \in Aut(F)$ and $\sigma \in Aut(K/F)$. As above, we note that $\chi_{a,0,0,1,\tau} \lambda_\sigma$ acts as $\upsilon^+_{\alpha,\tau,\sigma}$ for any $a \in K^*$ with $\N(a) = \alpha$, and $\chi_{0,b,1,0,\tau} \lambda_\sigma$ acts as $\upsilon^-_{\alpha,\tau,\sigma}$ for any $b \in K^*$ with $\N(b) = \alpha$. Hence there is an action isomorphism between $G_{\mathcal J}$ and $\Upsilon$, completing the proof.
\end{proof}

\section {Two-Dimensional Andr\'e Planes}

Now that we have our basic machinery in place, we begin our enumeration with the $n=1$ case, two-dimensional Andr\'e planes. Refreshing terminology, let $\S = \{J(\infty)\} \cup \{J(k):k \in K\}$ be a regular spread of $PG(3,q)$, $q >2$ a prime power, and let $\L = \{\J_f: k \in F^*\}$ be a linear set of reguli in $\S$. Letting ${\mathcal I}$ vary over all subsets of $F^*$, we can create every two-dimensional Andr\'e plane of order $q^2$ with the set of spreads $$\AI = \{J(\infty),J(0)\} \cup \bigcup_{f \in {\mathcal I}} \J^q_f \cup \bigcup_{f \in F^* \setminus {\mathcal I}} \J_f$$

We call the size of the set ${\mathcal I}$ the {\em index} of the Andr\'e spread. An Andr\'e plane of index either 0 or $q-1$ is obviously regular. Albert~\cite{albert} has shown that any plane obtained by switching a single regulus in a regular spread is necessarily the Hall plane, thus any Andr\'e plane of index either 1 or $q-2$ is a Hall plane. In what follows, we exclude these cases from consideration.

Since we are interested in sorting isomorphism classes, we may further restrict our attention to Andr\'e spreads of index at most $\frac{q-1}{2}$. Under the collineation $\lambda_q$, $\AI$ is isomorphic to $\A_{F^* \setminus {\mathcal I}}$ so every equivalence class of isomorphic Andr\'e planes contains at least one representative of index at most $\frac{q-1}{2}$.

These constraints show that the only Andr\'e spreads of $PG(3,3)$ and $PG(3,4)$ are the regular spread and the Hall spread. Thus we may restrict out attention to Andr\'e spreads of $PG(3,q)$ with $q \ge 5$, and of index $2 \le n \le \frac{q-1}{2}$. We begin with the following lemma which describes how a regular spread can meet an Andr\'e spread.

\begin{lemma}
\label{alemma}
Let $\AI$ be an Andr\'e spread, with index $n$ satisfying $2 \le n \le \frac{q-1}{2}$ in $PG(3,q)$, $q \ge 5$. The regular spread $\S$ meets $\AI$ in $2+(q-1-n)(q+1) > 2q+2$ lines, $\S^q$ meets $\AI$ in $2+n(q+1) > 2q+2$ lines, and no other regular spread meets $\AI$ in more than $2q+2$ lines.
\end{lemma}

\begin{proof}
The intersection sizes of $\AI$ with $\S$ and $\S^q$ are simple consequences of the construction of $\AI$, and the lower bounds on those sizes are a simple consequence of the bounds on $n$. Let $\T$ be a regular spread distinct from $\S$ and $\S^q$. We can write $\AI$ as the union of two partial spreads $\AI = (\S \cap \AI) \cup (\S^q \cap \AI)$. Two distinct regular spreads can meet in at most $q+1$ lines, so $\T$ meets each of the two components of this union in at most $q+1$ lines, making the total size of the intersection at most $2q+2$.
\end{proof}

With this lemma in place, our key result to sort isomorphisms between two-dimensional Andr\'e planes is straightforward.

\begin{prop}
\label{andre}
Let $\AI$ and $\AH$ be Andr\'e spreads of $PG(3,q)$ derived from the same regular spread $\S$, with indices between $2$ and $\frac{q-1}{2}$, inclusive. Then $\AI$ and $\AH$ are isomorphic if and only if:
\begin{enumerate}
\item there is a collineation of $PG(2n+1,q)$ leaving $\S$ invariant that maps the set of reguli $\{\J_f:f \in {\mathcal I}\}$ onto the set of reguli $\{\J_f:f \in {\mathcal H}\}$; or
\item there is a collineation of $PG(2n+1,q)$ leaving $\S$ invariant that maps the set of reguli $\{\J_f:f \in {\mathcal I}\}$ onto the set of reguli $\{\J_f:f \in F^* \setminus {\mathcal H}\}$
\end{enumerate}
\end{prop}

\begin{proof}
The reverse direction is relatively straightforward: clearly if there is an automorphism of $\S$ that maps the set of reguli $\{\J_f:f \in {\mathcal I}\}$ onto the set of reguli $\{\J_f:f \in {\mathcal H}\}$, it is an explicit isomorphism from $\AI$ to $\AH$. If there is an automorphism $\psi$ of $\S$ that maps $\{\J_f:f \in {\mathcal I}\}$ to $\{\J_f:f \in F^* \setminus {\mathcal H}\}$, then $\psi \lambda_q$ maps $\AI$ to $\AH$, since $\lambda_q$ interchanges $\J_f$ and $\J^q_f$ for all $f \in F^*$, resulting in a spread derived from $\S$ with exactly the reguli in $\{\J_f:f \in {\mathcal H}\}$ reversed, namely $\AH$.

Now suppose $\AI$ and $\AH$ are isomorphic, with collineation $\phi$ mapping $\AI$ to $\AH$. By Lemma~\ref{alemma}, each of $\AI$ and $\AH$ meets $\S$ and $\S^q$ in more than $2q+2$ lines, and meet no other regular spreads in that many lines. Hence $\phi$ must either leave $\S$ and $\S^q$ invariant, or must interchange them. The former case forces $\AI$ and $\AH$ to have the same index, while the latter forces $\AI$ and $\AH$ to have indices summing to $q-1$. By hypothesis the indices of $\AI$ and $\AH$ are both at most $\frac{q-1}{2}$, hence this case only occurs if $\AI$ and $\AH$ both have index $\frac{q-1}{2}$.

Suppose first the isomorphism $\phi$ leaves $\S$ invariant. Then $\phi$ maps $\AI \cap \S$ to $\AH \cap \S$, and thus must map the set of reguli in $\{\J_f:f \in {\mathcal I}\}$ onto a set of reguli contained in the union of lines of the reguli in $\{\J_f:f \in {\mathcal H}\}$. Since $\{\J_f:f \in {\mathcal H}\}$ has at most $\frac{q-1}{2}$ reguli, by the pigeonhole principle, the image of each regulus in $\{\J_f:f \in {\mathcal I}\}$ under $\phi$ must meet some regulus in $\{\J_f:f \in {\mathcal H}\}$ in at least 3 lines, forcing it to be identical to one of those reguli. Hence every regulus in $\{\J_f:f \in {\mathcal I}\}$ must map under $\phi$ onto a regulus in $\{\J_f:f \in {\mathcal H}\}$, and the fact that the indices of the two spreads are the same forces $\phi$ to map the set of reguli $\{\J_f:f \in {\mathcal I}\}$ onto the set of reguli $\{\J_f:f \in {\mathcal H}\}$.

If the isomorphism $\phi$ interchanges $\S$ and $\S^q$, we know that the index of $\AI$ and $\AH$ is $\frac{q-1}{2}$. Consider the collineation $\psi = \lambda_q \phi$. For any $f \in {\mathcal I}$, $\J_f^{\lambda_q} = \J^q_f$ is a regulus in $\AI$ and also in $\S^q$. Applying $\phi$ to this regulus gives a regulus $J$, which lies in both $\AH = \AI^\phi$ and $\S = (\S^q)^\phi$. Thus $J$ must be a subset of the union of the reguli $\{\J_f:f \in F^* \setminus {\mathcal H}\}$. But since the index of $\AH$ is $\frac{q-1}{2}$, this union is of $\frac{q-1}{2}$ reguli, and as before $J$ must share at least 3 lines with one of the $\J_f$ for $f \in F^* \setminus {\mathcal H}$, and thus must be equal to that regulus. Therefore, $\psi$ is a collineation of $PG(2n+1,q)$ leaving $\S$ invariant that maps the set of reguli $\{\J_f:f \in {\mathcal I}\}$ onto the set of reguli $\{\J_f:f \in F^* \setminus {\mathcal H}\}$, completing the proof.
\end{proof}

Suppose $\{\J_f:f \in {\mathcal I}\}$ and $\{\J_f:f \in {\mathcal H}\}$ are two sets of reguli with size between 2 and $\frac{q-1}{2}$ inclusive, in a linear set $\L$  of a regular spread $\S$, and we wish to determine if they generate isomorphic Andr\'e spreads, and thus isomorphic Andr\'e planes. From Bruck~\cite{bruck} $\L$ is the only linear set of reguli containing $\{\J_f:f \in {\mathcal I}\}$, and also the only linear set containing $\{\J_f:f \in {\mathcal H}\}$. Moreover, since ${\mathcal I}$ and ${\mathcal H}$ have less than $\frac{q-1}{2}$ elements, the complementary sets of reguli $\{\J_f:f \in F^* \setminus {\mathcal I}\}$ and $\{\J_f:f \in F^* \setminus {\mathcal H}\}$ also have at least two elements, and thus are only contained in the linear set $\L$. Thus any collineation of $PG(2n+1,q)$ leaving $\S$ invariant which maps $\{\J_f:f \in {\mathcal I}\}$ to $\{\J_f:f \in {\mathcal H}\}$ or its complement in $\L$ must leave $\L$ invariant. This allows us to apply Theorem~\ref{lingroups}, parts 1 and 2, from which we find that the spreads $\AI$ and $\AH$ are isomorphic if and only if there exist $\alpha \in F^*$ and $\tau \in Aut(F)$ such that ${\mathcal I}^{\xi^\pm_{\alpha,\tau}} = {\mathcal H}$, for some choice of sign, or in the case where $|{\mathcal I}|=\frac{q-1}{2}$, there exist $\alpha \in F^*$ and $\tau \in Aut(F)$ such that ${\mathcal I}^{\xi^\pm_{\alpha,\tau}} = F^* \setminus {\mathcal H}$, for some choice of sign.

For $q=5$, this problem is tractable by hand, since the only case we have to deal with is that of index $2 = \frac{q-1}{2}$. One possible set of size 2 is $\{\J_1,\J_2\}$. There are no nontrivial automorphisms of $GF(5)$, so we see this set maps to $\{\J_2,\J_4\}$, $\{\J_1,\J_3\}$ and $\{\J_3,\J_4\}$ under $\xi^+_{\alpha,1}$ for $\alpha \in GF(5) \setminus\{0\}$. The inversion map $\xi^-_{1,1}$ interchanges $\{\J_1,\J_2\}$ with $\{\J_1,\J_3\}$ and $\{\J_2,\J_4\}$ with $\{\J_3,\J_4\}$, and the complement of each of these sets is already represented, so these four sets form one orbit under $\xi$. It is easy to see that the remaining two pairs, $\{\J_1,\J_4\}$ and $\{\J_2,\J_3\}$, form a second orbit. Hence there are two Andr\'e planes of order 25 with index 2. This is validated by the enumeration of all translation planes of order 25 by Czerwinski and Oakden~\cite{czoak}, where 5 subregular spreads are found in $PG(3,5)$: the regular spread, a Hall spread, a subregular spread from a non-linear triple, and two Andr\'e planes, one of which is in fact the regular nearfield plane of order 25. Using MAGMA~\cite{magma}, we have automated this calculation for small $q$; the code implementing this enumeration is in Appendix A. For some small values of $q$, all of the two-dimensional Andr\'e planes, excepting the Desarguesian and Hall planes, can be obtained from the sets of $\J$s in Table~\ref{andreenum}.

\begin{table}
\label{andreenum}
\caption{Enumeration of Andr\'e planes of order $q^2$ with index at least 2 for small $q$}
\begin{tabular}{ |c|c|lr| }
\hline
$q$ & {\bf Number} & \multicolumn{2}{c}{\bf Representatives}\\
\hline\hline
5 & 2 & $\{1,2\},\{1,4\}$&\\
\hline
7 & 6 & $\{3,4\},\{4,6\},\{3,5\},\{3,5,6\},\{1,3,5\},\{1,3,6\}$&\\
\hline
8 & 3 & $\{\omega^3,\omega^5\},\{\omega^2,\omega^3,\omega^5\},\{\omega^2,\omega^4,\omega^6\}$&($\omega^3=\omega+1$)\\
\hline
9 & 12 & $\{1,2\},\{1,\tau^6\},\{\tau^2,\tau^3\}$&($\tau^2 = \tau+1$)\\
& & $\{\tau,\tau^5,\tau^7\},\{\tau,\tau^3,2\},\{1,\tau^3,2\},\{1,\tau,\tau^2\}$&\\
& & $\{1,\tau^2,2,\tau^6\}, \{1,\tau^3,2,\tau^7\},\{1,\tau,\tau^3,2\},\{\tau,2,\tau^6,\tau^7\},\{\tau^3,\tau^5,\tau^6,\tau^7\}$&\\
\hline
\end{tabular}
\end{table}

This method produces representative sets of reguli for each Andr\'e plane, but bogs down as $q$ increases, due to the need to create the actual subsets of reguli for each index. If we are only interested in the number of distinct Andr\'e planes with a given index, we can appeal to Burnside's lemma to count the number of orbits under the group $\Xi$, with the exception of the Andr\'e planes of index $\frac{q-1}{2}$, where complementation is again a confounding factor. Let us consider this case in more detail.

Let $B$ be the set of all subsets of $F^*$ of size $\frac{q-1}{2}$. Each element of $\Xi$ induces an action on $B$ through its action on the individual elements of $\{\J_f:f \in F^*\}$. Complementation is also a permutation on $B$, though it does not act element-wise; therefore there exists a permutation $\gamma$ acting on $B$ such that $b^\gamma = \overline{b}$ for all $b \in B$. Hence we can perform the same Burnside counting in this case, but we have to use the group $\Xi^*$ generated by the elements of $\Xi$ and $\gamma$. This turns out not to be as taxing as one might fear, based on the following result.

\begin{prop}
\label{burnside1}
Let $F = GF(q)$, $q > 2$ a prime power, and let $\Xi$ be the group of transformations $\xi^\pm_{\alpha,\tau}$ acting on $F^*$, with $\alpha \in F^*$ and $\tau \in Aut(F)$. Consider the induced group action of $\Xi$ on the set $B$ of subsets of $F^*$ of size $\frac{q-1}{2}$, and let $\Xi^*$ be the group generated by the permutations in $\Xi$ and $\gamma$, acting on $B$. Then:
\begin{enumerate}
\item For all $\xi \in \Xi$, $\xi$ and $\gamma$ commute;
\item $|\Xi^*|$ = $2|\Xi|$;
\item $\xi \gamma$ fixes some $b \in B$ if and only if all orbits of $\xi$, when acting on $F^*$, have even length; and
\item if all orbits of $\xi$, when acting on $F^*$, have even length, then $\xi \gamma$ fixes $2^{o(\xi)}$ elements of $B$, where $o(\xi)$ is the number of orbits of $\xi$ when acting on $F^*$.
\end{enumerate}
\end{prop}

\begin{proof}
The first two statements are clear, since the action of $\xi$ on $B$ preserves complementation, namely if $a^\xi = b$ then $\overline{a}^\xi = \overline{b}$, and $\gamma$ has order 2 as a permutation. Suppose now that for some $b \in B$ we have $b^{\xi\gamma} = b$. For each element $f \in b$, we have $f^\xi \notin b$ in the action of $\xi$ on $F^*$. Since $|b| = |\overline{b}| = \frac{q-1}{2}$, this implies for all $f \in \overline{b}$, the preimage of $f$ under $\xi$ must be in $b$, hence for all $f \in \overline{b}$, we have $f^\xi \in b$. So if $O$ is any orbit of $\xi$ in its action on $F^*$, its elements must alternate being in $b$ and not in $b$, forcing each such orbit to have even length. Conversely, suppose every orbit of $\xi \in \Xi$ in its action on $F^*$ has even length. Pick one element from each orbit of $\xi$, and let $b$ be the union of the orbits of these elements under $\xi^2$. This set contains exactly half of the elements of $F^*$, and $\xi$ maps each element of $b$ to an element not in $b$, hence $\xi \gamma$ fixes $b$. This also shows the fourth claim, since each orbit of $\xi$ in its action on $F^*$ splits into two parts, each of which can be picked independently of all other orbits, to add into a set $b$ fixed by $\xi \gamma$.
\end{proof}

With this proposition in place, we have developed code in MAGMA to implement this Burnside counting; this code can be seen in Appendix B. For some small values of $q$, we obtain the counts of Andr\'e planes with a given index in Table~\ref{andrecount}.

\begin{table}
\label{andrecount}
\caption{Numbers of Andr\'e planes of order $q^2$, by index}
\begin{tabular}{ |c||cccccccccccc| }
\hline
\multirow{2}{6pt}{$q$} & \multicolumn{12}{c}{\bf Index}\vline\\
& 2 & 3 & 4 & 5 & 6 & 7 & 8 & 9 & 10 & 11 & 12 & 13\\
\hline\hline
5 & 2&&&&&&&&&&&\\
\hline
7 & 3 & 3&&&&&&&&&&\\
\hline
8 & 1 & 2&&&&&&&&&&\\
\hline
9 & 3 & 4 & 5&&&&&&&&&\\
\hline
11 & 5 & 8 & 16 & 13&&&&&&&&\\
\hline
13 & 6 & 12 & 29 & 38 & 35&&&&&&&\\
\hline
16 & 3 & 7 & 18 & 34 & 54 & 66&&&&&&\\
\hline
17 & 8 & 21 & 72 & 147 & 280 & 375 & 257&&&&&\\
\hline
19 & 9 & 27 & 104 & 252 & 561 & 912 & 1282 & 765&&&&\\
\hline
23 & 11 & 40 & 195 & 621 & 1782 & 3936 & 7440 & 11410 & 14938 & 8359&&\\
\hline
25 & 8 & 30 & 143 & 487 & 1517 & 3741 & 7934 & 13897 & 20876 & 26390 & 14632&\\
\hline
27 & 5 & 20 & 112 & 434 & 1532 & 4264 & 10145 & 20121 & 34291 & 49668 & 62220 & 33798\\
\hline
\end{tabular}
\end{table}
 
\section{Higher-Dimensional Andr\'e Planes}

The most significant difference between the higher-dimensional Andr\'e planes and the two-dimensional case is the presence of multiple replacements for a norm surface, which in turn yields additional regular spreads which can have a substantial intersection with an Andr\'e spread. This also makes the concept of index ambiguous, because just knowing that a norm surface is reversed is not enough information to determine an Andr\'e spread; we need to know which replacement is chosen. To this end, we define an indicator function $\I:F^* \rightarrow Aut(K/F)$, from which we obtain the Andr\'e spread $\AI = \{J(\infty),J(0)\} \cup \bigcup_{f \in F^*} \J^{\I(f)}_f$. Note that there are exactly $(n+1)^{q-1}$ such indicator functions, and they define all of the possible Andr\'e spreads.

Our first lemma is an analog of Lemma~\ref{alemma} for the higher-dimensional case.

\begin{lemma}
\label{3alemma}
Let $\AI$ be an Andr\'e spread of $PG(2n+1,q)$, $n > 1$, $q > 2$, defined via indicator function $\I$. Then the only regular spreads that meet $\AI$ in at least $q^2$ spaces are the regular spreads $\S^\sigma$ for $\sigma \in Aut(K/F)$ in the range of $\I$. Moreover, a spread $\S^\sigma$ meets $\AI$ in either exactly two spaces, or more than $q^2$.
\end{lemma}

\begin{proof}
Let $\AI$ be as in the lemma statement, and let $\T$ be a regular spread distinct from the $\S^\sigma$. Then for each $f \in F^*$, $\T$ meets the set $\J^{\I(f)}_f \cup \{(J(0),J(\infty)\} \subset \S^{\I(f)}$ in at most $q+1$ spaces, hence the total number of spaces in $\T \cap \AI$ is at most $(q+1)(q-1) = q^2-1$. For any field automorphism $\tau \in Aut(K/F)$, if $\tau = \I(f)$ for some $f \in F^*$, then $\S^\tau$ contains at least $2+\frac{q^{n+1}-1}{q-1} > q^2$ spaces of $\AI$; otherwise $\S^\tau$ meets $\AI$ in just $J(\infty)$ and $J(0)$.
\end{proof}

We can now prove the key result we need to sort isomorphism classes of higher-dimensional Andr\'e planes. It is not as strong as the corresponding result for $n=1$, requiring more work in enumeration, but it is an improvement over having to work directly with spreads of $PG(2n+1,q)$.

\begin{prop}
\label{lin3}
Let $\AI$ and $\AH$ be non-regular Andr\'e spreads of $PG(2n+1,q)$ obtained from the regular spread $\S$ with indicator functions $\I$ and $\H$. Then any isomorphism $\psi$ between $\AI$ and $\AH$ must leave $\{J(\infty),J(0)\}$ invariant. Moreover, $\psi$ maps the collection of sets of subspaces ${\mathcal J} = \{\J^\sigma_f:\sigma \in Aut(K/F), f \in F^*\}$ onto itself.
\end{prop}

\begin{proof}
Suppose $\AI$ and $\AH$ are isomorphic, with isomorphism $\psi$ a collineation in $PG(2n+1,q)$ mapping $\AI$ onto $\AH$. Since $\psi$ is an isomorphism, it must map the set of regular spreads $\S_\I$ meeting $\AI$ in more than $q^2$ points to the set of regular spreads $\S_\H$ meeting $\AH$ in more than $q^2$ points, hence $\psi$ must map the intersection of the $\S_\I$ onto the intersection of the $\S_\H$. By Lemma~\ref{3alemma}, $\S_\I$ and $\S_\H$ are both subsets of $\{\S^\sigma:\sigma \in Aut(K/F)\}$ and since $\AI$ and $\AH$ are not regular, $\S_\I$ and $\S_\H$ each contain at least two spreads, implying the intersection is exactly $\{J(\infty),J(0)\}$, which must be left invariant by $\psi$.

Suppose there exists $\sigma \in Aut(K/F)$ such that $(\S^\sigma)^\psi = \S^\sigma$; without loss of generality, we may assume $\sigma$ is the identity. Then $\psi$ is a collineation of $PG(2n+1,q)$ that leaves $\S$ invariant, leaves the set $\{J(\infty),J(0)\}$ invariant, and thus the linear set $\L = \{\J_f:f \in F^*\}$ invariant as well. By Theorem~\ref{lingroups}, this implies $\psi \in G_\L$ and thus $\psi \in G_{\mathcal J}$, showing $\psi$ leaves the set ${\mathcal J}$ invariant.

Suppose now there is no $\sigma \in Aut(K/F)$ such that $\psi$ leaves $\S^\sigma$ invariant. Since the set $\S_\I$ of regular spreads meeting $\AI$ in more than $q^2$ points has at least two members, there must be some spread in $\S_\I$ that meets $\AI$ in at least one, but at most $\frac{q-1}{2}$ norm surfaces, plus $\{J(\infty),J(0)\}$. Without loss of generaity, we may assume this spread is $\S$. Since $\S \in \S_\I$, we know $(\S)^\psi \in \S_\H$ and thus $(\S)^\psi = \S^\tau$  for some $\tau \in Aut(K/F)$.

For each norm surface $\J = \J_f \subset \AI \cap \S$, $\J^\psi$ is contained in the union of at most $\frac{q-1}{2}$ norm surfaces $\J^{\tau}_{f'} \subset \AH \cap \S^\tau$; $\J^\psi$ cannot intersect $\{J(\infty),J(0)\}$ as these two spaces are left invariant by $\psi$. But by Lemma~\ref{intersection}, $\J^\psi$ must be identical with one of the $\J^{\tau}_{f'}$, for otherwise $\J^\psi$ could only contain at most $\frac{q-1}{2} \times 2\frac{q^n-1}{q-1} < \frac{q^{n+1}-1}{q-1}$ spaces. By Proposition~\ref{linear}, a norm surface belongs to only one linear set of norm surfaces in a regular spread, so this implies that $\psi$ must map the set of norm surfaces $\{\J_f:f \in F^*\}$ in $\S$ onto the set of norm surfaces $\{\J^\tau_{f}:f \in F^*\}$ in $\S^\tau$.

Recall the collineation $\lambda_\tau$ of $PG(2n+1,q)$ defined via $(x,y)\lambda_\tau = (x,y^\tau)$. Clearly $\lambda_\tau$ leaves each of $J(\infty)$ and $J(0)$ invariant, and maps $J^\tau(k)$ onto $J(k)$ for all $k \in K^*$. Thus $\psi \lambda_\tau$ leaves $\S$ invariant, and leaves the set $\{J(\infty),J(0)\}$ invariant. Using the same trick as before with Theorem~\ref{lingroups}, this forces $\psi \lambda_\tau = \chi_{a,b,c,d,\rho}$ with either $b=c=0$ or with $a=d=0$, and a similar calculation to the above shows that $\psi$ maps the collection of sets of subspaces ${\mathcal J}=\{\J^\sigma_f:\sigma \in Aut(K/F), f \in F^*\}$ onto itself.
\end{proof}

In light of this Proposition, we can use the group $\Upsilon$ of Theorem~\ref{lingroups} to sort isomorphism between Andr\'e spreads. Letting $\I$ and $\H$ be indicator functions of two Andr\'e spreads $\AI$ and $\AH$, this proposition shows that there exists an isomorphism between $\AI$ and $\AH$ if and only if there exists $\upsilon \in \Upsilon$ such that $(\J^{\I(f)}_f)^\upsilon = \J^{\H(f')}_{f'}$ for some $f' \in F^*$. From a computational perspective, we can represent the indicator functions as sets of ordered pairs $\{(f,\I(f)):f \in F^*\}$ and extending the action of $\upsilon$ naturally to these ordered pairs, we see that $\AI$ and $\AH$ are isomorphic if and only if there exists $\upsilon \in \Upsilon$ such that $\{(f,\I(f)):f \in F^*\}^\upsilon = \{(f,\H(f)):f \in F^*\}$.

\begin{table}
\label{t3d}
\caption{Enumeration of higher-dimensional Andr\'e planes for some small $q$}
\begin{tabular}{ |c|c|c|c|lr| }
\hline
$n$ & $q$ & Order & {\bf Number} & \multicolumn{2}{c}{\bf Representatives}\\
\hline\hline
2 & 3 & 27 & 1 & $\{(1,1),(2,3)\}$&\\
\hline
2 & 4 & 64 & 2 & $\{(1,1),(\omega,1),(\omega^2,4)\}$,$\{(1,1),(\omega,4),(\omega^2,16)\}$ & $(\omega^2 = w+1)$\\
\hline
2 & 5 & 125 & 6 & $\{(1,1),(2,1),(3,1),(4,5)\}$, $\{(1,1),(2,1),(3,1),(4,25)\}$&\\
&&&& $\{(1,1),(2,5),(3,5),(4,1)\}$, $\{(1,1),(2,5),(3,1),(4,5)\}$&\\
&&&& $\{(1,1),(2,1),(3,5),(4,25)\}$, $\{(1,1),(2,1),(3,25),(4,5)\}$&\\
\hline
3 & 3 & 81 & 2 & $\{(1,1),(2,3)\}$,$\{(1,1),(2,9)\}$&\\
\hline
3 & 4 & 256 & 3 & $\{(1,1),(\omega,1),(\omega^2,4)\}$,$\{(1,1),(\omega,1),(\omega^2,16)\}$,&$(\omega^2 = w+1)$\\
&&&& $\{(1,1),(\omega,4),(\omega^2,16)\}$&\\
\hline
\end{tabular}
\end{table}

As before, this process is tractable for small cases by hand; let us consider the Andr\'e planes of order 27, whence $n=2$ and $q=3$. In this case, there are only 9 Andr\'e spreads. Starting with $\{(1,1),(2,1)\}$, which is just the regular spread $\S$, we see this set is preserved under both multiplication by $-1$ and inversion, so the only other spreads in its orbit are $\{(1,3),(2,3)\}$ and $\{(1,9),(2,9)\}$, namely $\S^3$ and $\S^9$. A non-Desarguesian Andr\'e plane is represented by the set $\{(1,1),(2,3)\}$. Applying $\upsilon^+_{1,1,3}$ twice shows that $\{(1,3),(2,9)\}$ and $\{(1,9),(2,1)\}$ represent isomorphic spreads, and applying $\upsilon^+_{2,1,1}$ maps $\{(1,1),(2,3)\}$ to $\{(2,1),(1,3)\}$. Thus all six of the non-regular Andr\'e spreads are isomorphic, and there is just one non-Desarguesian Andr\'e plane of order 27.

Representing our sets $\{(f,\I(f)):f \in F^*\}$ as a sequence of field automorphisms by defining a consistent ordering of the elements of $F^*$, we have implemented this sorting algorithm in MAGMA. Note that $\Upsilon$ can be generated by four elements: $\upsilon^+_{\omega,1,1}$, $\upsilon^-_{1,1,1}$, $\upsilon^+_{1,p,1}$, and $\upsilon^+_{1,1,q}$, where $p$ is the characteristic of $F$, and we use MAGMA's group generation algorithms to create the entire group. Table 3 provides indicator functions for all non-Desarguesian  Andr\'e planes for some small orders with given $n$ and $q$.

For larger values of $q$, we can again count Andr\'e plane non-constructively using Burnside's lemma and the group $\Upsilon$ of Theorem~\ref{lingroups}. The situation in higher dimensions is messier than that with $n=1$, as we need to break into several cases to count the number of Andr\'e spreads fixed by each element of $\Upsilon$. The cause of the difficulty is the inversion map $(x,y)^\upsilon = (y,x)$; when $n=1$ this map preserves both the spreads $\S$ and $\S^q$, but for larger $n$ this map may interchange some of the replacements for norm surfaces, making the analysis more intricate.

Let us start with the easy case, namely $\upsilon = \upsilon^+_{\alpha,\tau,\sigma}$, and define $\hat{\upsilon}:F^* \rightarrow F^*$ be defined via $f^{\hat{\upsilon}} = \alpha f^\tau$. If $\AI$ is an Andr\'e spread left invariant by $\upsilon$, and $\J^\rho_f \in \AI$, then from the definition of $\upsilon$, we know $$(\J^\rho_f)^\upsilon = \J^{\rho\sigma^{-1}}_{\alpha f^{\tau}} = \J^{\rho\sigma^{-1}}_{f^{\hat{\upsilon}}}$$

For any $a \in F^*$, let $\ell$ be the length of its orbit under $\hat{\upsilon}$. This implies that $(\J^\rho_a)^{\upsilon^\ell} = \J^\mu_a$ for some automorphism $\mu \in Aut(K/F)$. Since $\upsilon$ leaves $\AI$ invariant and $\AI$ contains only one norm surface associated with norm $a$, we must have $\mu = \rho$. We see that $\mu = \rho \sigma^{-\ell}$, so $\mu = \rho$ if and only if $\ell$ is divisible by the order of $\sigma$ as a field automorphism.

Therefore $\upsilon$ fixes an Andr\'e spread $\AI$ if and only if all of the orbit lengths of $\hat{\upsilon}$ are divisible by the order of $\sigma$. If this does occur, then much as in the $n=1$ case we can count the number of Andr\'e spreads fixed by $\upsilon$. For one member $a$ of each orbit under $\hat{\upsilon}$, we may select an arbitrary $\rho_a \in Aut(K/F)$ in $n+1$ ways; then the union of the orbits of $\J^{\rho_a}_a$ under $\upsilon$, for each $a$, form an Andr\'e spread fixed by $\upsilon$. Hence $\upsilon$ fixes $(n+1)^{o(\hat{\upsilon})}$ Andr\'e spreads $\AI$, where $o(\hat{\upsilon})$ is the number of orbits of $\hat{\upsilon}$.

If $\upsilon = \upsilon^-_{\alpha,\tau,\sigma}$, the situation is slightly more complex. Defining $\hat{\upsilon}$ as above, we have $$(\J^\rho_f)^\upsilon = \J^{\rho^{-1}\sigma^{-1}}_{\alpha f^{-\tau}} = \J^{\rho^{-1}\sigma^{-1}}_{f^{\hat{\upsilon}}}$$ and $$(\J^\rho_f)^{\upsilon^2} = \J^{\rho}_{f^{\hat{\upsilon}^2}}$$

Let $a \in F^*$. If the orbit length of $a$ under $\hat{\upsilon}$ is even, then just like before we can pick $\rho \in Aut(K/F)$ arbitrarily, and the orbit of $\J^\rho_a$ under $\upsilon$ will form part of an Andr\'e spread. But if the orbit length of $a$ under $\hat{\upsilon}$ is odd, we must have $\rho = \rho^{-1} \sigma^{-1}$, or $\rho^2 = \sigma^{-1}$, which is eminently feasible since $Aut(K/F)$ is cyclic of order $n+1$. If $n$ is even, then for every $\sigma \in Aut(K/F)$ there is a unique $\rho$ such that $\rho^2 = \sigma^{-1}$, so any Andr\'e spread fixed by this $\upsilon$ must contain $\J^\rho_f$ for all $f$ in the orbit of $a$ under $\hat{\upsilon}$. If $n$ is odd, then for half of the $\sigma \in Aut(K/F)$, there is no $\rho$ such that $\rho^2 = \sigma^{-1}$, and those $\upsilon$ fix no Andr\'e spreads. For the remaining half there are two options for $\rho$, giving two orbits of $\J^\rho_a$ which can be in an Andr\'e spread fixed by $\upsilon$.

Based on this analysis, we have developed code in MAGMA to implement this counting procedure; this code appears as Appendix D. Table~\ref{c3d} shows the number of non-isomorphic Andr\'e planes, not including the Desarguesian plane, for some small values of $n$ and $q$. Note that one cannot say for a given $n$ and $q$ that every Andr\'e plane derived from an Andr\'e spread of $PG(2n+1,q)$ has dimension $n+1$ over a full kernel $GF(q)$; the regular spread belies this notion; we do not attempt to sort isomorphism between Andr\'e planes for different $n$ here.

\begin{table}
\label{c3d}
\caption{Number of non-Desarguesian Andr\'e planes obtained from Andr\'e spreads of $PG(2n+1,q)$}
\begin{tabular}{ |c||ccccccccc| }
\hline
\multirow{2}{6pt}{$q$} & \multicolumn{9}{c}{$n$}\vline\\
& 2 & 3 & 4 & 5 & 6 & 7 & 8 & 9 & 10\\
\hline\hline
3 & 1 & 2 & 2 & 3 & 3 & 4 & 4 & 5 & 5\\
\hline
4 & 2 & 3 & 4 & 6 & 7 & 9 & 11 & 13 & 15\\
\hline
5 & 6 & 15 & 23 & 40 & 57 & 86 & 114 & 157 & 200\\
\hline
7 & 31 & 112 & 300 & 729 & 1503 & 2902 & 5134 & 8651 & 13795\\
\hline
8 & 25 & 114 & 402 & 1160 & 2877 & 6350 & 12804 & 24012 & 42445\\
\hline
\end{tabular}
\end{table}

\section{Conclusion}
The early history of finite projective planes provided a handful of infinite families, as well as some sporadic examples with particularly interesting automorphisms. Primarily in the 1990s and into the 2000s, there was an explosion in both construction techniques for new planes, as well as in computational results classifying the finite translation planes of orders 16, 25, 27 and 49. Ironically, this mass of data seems to have provided a disincentive to continued work. On the one hand, the plethora of new construction techniques gives the impression that most translation planes are ``known", while the computational results suggest that a complete classification of finite projective planes, let alone translation planes, is infeasible.

Enumeration of existing families of translation planes is not groundbreaking work, but it does represent a way forward in the theory of finite projective planes. The author~\cite{gen25} has taken a step in this direction, by correlating planes created by known construction techniques against the list of translation planes of order 25. This analysis shows that there is a translation plane, designated B8 by Czerwinski and Oakden~\cite{czoak}, which cannot be obtained from existing techniques. It may be sporadic, or it may be from an infinite family yet to be found; regardless, it provides an interesting question for further research. We view this work as a contribution toward ``clearing out the underbrush", hopefully helping to reveal trailheads for the next paths forward in the theory of finite projective planes.

\bibliographystyle{plain}

\section*{Appendix A: Two-Dimensional Andr\'e Enumeration Code}
\begin{Verbatim}[fontsize=\tiny]
/*Set up basic structures to implement group*/
q:=7; F<w>:=GF(q); F0:=SetToSequence(Set(F) diff {0}); Sq:=Sym(q-1);

/*Create multiplication and inverse permutations.  Group creation will ensure we get all.*/
phi:=[Sq![Position(F0,a*F0[i]):i in {1..#F0}]:a in F0] cat [Sq![Position(F0,1/F0[i]):i in {1..#F0}]];
	 
/*Make sure to add field automorphism if q is not prime.*/
if not(IsPrime(q)) then Append(~phi,Sq![Position(F0,F0[i]^(Divisors(q)[2])):i in {1..#F0}]); end if;
G:=sub<Sq|phi>;

/*Iterate through subregular indices, calculating orbits*/
for i in {2..Floor((q-1)/2)} do
  S:=Subsets({1..#F0},i);
  GS:=GSet(G,S);
  O:=Orbits(G,GS);
  reps:=[];
  /*This iteration coalesces complements. Really only needed when 2i+1 = q, but causes no harm otherwise*/
  for j in {1..#O} do
    if forall{x:x in reps|{1..#F0} diff x notin O[j]} then
	  Append(~reps,Rep(O[j]));
	end if;
  end for;
  print i,[{F0[j]:j in x}:x in reps];
end for;
\end{Verbatim}

\section*{Appendix B: Two-Dimensional Andr\'e Counting Code}
\begin{Verbatim}[fontsize=\tiny]
q:=11; F<w>:=GF(q); F0:=SetToSequence(Set(F) diff {0}); Sq:=Sym(q-1);

/*Create multiplication and inverse permutations. Group creation will ensure we get all.*/
phi:=[Sq![Position(F0,a*F0[i]):i in {1..#F0}]:a in F0] cat [Sq![Position(F0,1/F0[i]):i in {1..#F0}]];

/*Make sure to add field automorphism if q is not prime.*/
if not(IsPrime(q)) then Append(~phi,Sq![Position(F0,F0[i]^(Divisors(q)[2])):i in {1..#F0}]); end if;
 G:=sub<Sq|phi>;

/*Determine cycle structure of each group element, and store it as a sequence of number of cycles of each length.
   There will be repeats, so we will store by frequency.*/
gcyc:=[]; gfrq:=[];
for x in G do
  x1:=CycleStructure(x);
  gtmp:=[0:j in {1..q-1}];
  for k in {1..#x1} do
    gtmp[x1[k][1]]:=x1[k][2];
  end for;
  if Position(gcyc,gtmp) eq 0 then
    Append(~gcyc,gtmp);
	Append(~gfrq,1);
  else
    gfrq[Position(gcyc,gtmp)]+:=1;
  end if;
end for;

for i in {2..Floor((q-1)/2)} do
  Gfix:=0;
  p1:=Partitions(i);
  for p in p1 do
	prt:=[#{k:k in {1..#p}|p[k] eq j}:j in {1..i}];
    /* Convert partitions of i into sequences counting the number of occurrences of each integer */
	for j in {1..#gcyc} do
	  coll:=1;
	  for k in {1..#prt} do
	    if prt[k] ne 0 then coll*:=Binomial(gcyc[j][k],prt[k]); end if;
	  end for;
	  Gfix+:=gfrq[j]*coll;
	end for;
  end for;
  /*Extra group elements in maximal case*/
  if (2*i+1 eq q) then
    for j in {1..#gcyc} do
	  if forall{m:m in {1..#gcyc[j]}|IsEven(m) or gcyc[j][m] eq 0} then
 	    Gfix+:=2^(&+(gcyc[j]))*gfrq[j];
      end if;
    end for;
    print i,Gfix/(2*Order(G));
  else
    print i,Gfix/Order(G);
  end if;
end for;
\end{Verbatim}

\section*{Appendix C: Higher-Dimensional Andr\'e Enumeration Code}
\begin{Verbatim}[fontsize=\tiny]
n:=3; q:=5; F:=GF(q);

/*Find a primitive element omega*/
omega:=Rep({f:f in F|f ne 0 and Order(f) eq q-1}); Fstar:=[omega^i:i in {0..q-2}];

/*Generate all sequences of automorphisms of length q-1*/
So:=[[q^i]:i in {0..n}];
for j in {1..q-2} do
  S:=[s cat [q^i]:s in So,i in {0..n}];
  So:=S;
end for;
G:=Sym(#S);
p1:=[]; p2:=[]; p3:=[]; p4:=[];

/*upsilon+{omega,1,1)...the ordering on Fstar implements this by rotation of the sequence.*/
for i in {1..#S} do Append(~p1,Position(S,Rotate(S[i],1))); end for;
g1:=G!p1;

/*upsilon-{1,1,1}*/
iseq:=[Position(Fstar,Fstar[i]^(-1)):i in {1..#Fstar}];
for i in {1..#S} do Append(~p2,Position(S,[(q^(n+1) div S[i][iseq[j]]) mod (q^(n+1)-1):j in {1..q-1}])); end for;
g2:=G!p2;

/*upsilon+{1,p,1}*/
if not(IsPrime(q)) then
  p:=Characteristic(F);
  iseq:=[Position(Fstar,Fstar[i]^p):i in {1..#Fstar}];
  for i in {1..#S} do Append(~p3,Position(S,[S[i][iseq[j]]:j in {1..q-1}])); end for;
  g3:=G!p3;
else g3:=G!1; end if;

/*upsilon+{1,1,q}*/
for i in {1..#S} do Append(~p4,Position(S,[S[i][j]*q mod (q^(n+1)-1):j in {1..q-1}])); end for;
g4:=G!p4;

H:=sub<G|g1,g2,g3,g4>;
O:=Orbits(H); #O;
\end{Verbatim}

\section*{Appendix D: Higher-Dimensional Andr\'e Counting Code}
\begin{Verbatim}[fontsize=\tiny]
F<w>:=GF(q); Fstar:=SetToSequence(Set(F) diff {0}); Sq:=Sym(q-1);
Gfix:=0;
for alpha in Fstar do
  for tau in Prune(Divisors(q)) do /*Prune removes q*/
    /* Do the plus version first*/
    phi:=Sq![Position(Fstar,alpha*Fstar[i]^tau):i in {1..#Fstar}];
    x:=CycleStructure(phi);
    for o in Divisors(n+1) do /*Use cyclic structure of Aut(K/F)*/
      Sfix:=1;
      for i in {1..#x} do
        if (x[i][1] mod o) eq 0 then
          Sfix*:=(n+1)^x[i][2];
        else
          Sfix:=0;
          break;
        end if;
      end for;
      Gfix+:=Sfix*EulerPhi(o); /*phi(o) elements of order o*/
    end for;

    /*Now the minuses*/
    phi:=Sq![Position(Fstar,alpha/(Fstar[i]^tau)):i in {1..#Fstar}];
    x:=CycleStructure(phi);
    for o in Divisors(n+1) do
      Sfix:=1;
      for i in {1..#x} do
        if IsEven(x[i][1]) then
          Sfix*:=(n+1)^x[i][2];
        else
          if IsEven(n) then
            Sfix*:=1;
          else
            if IsEven((n+1) div o) then
              Sfix*:=2^x[i][2];
            else
              Sfix:=0;
              break;
            end if;
          end if;
        end if;
      end for;
      Gfix+:=Sfix*EulerPhi(o); /*phi(o) elements of order o*/
    end for;
  end for;
end for;
print Gfix/(2*(q-1)*#(Prune(Divisors(q)))*(n+1))-1;
\end{Verbatim}
\end{document}